\documentclass{scrartcl}

\usepackage{hyperref}
\usepackage{lmodern}
\usepackage[T1]{fontenc}
\usepackage{amsmath}
\usepackage{upgreek}
\usepackage{amssymb}
\usepackage{amsthm}
\usepackage[backend=bibtex]{biblatex}
\usepackage{enumitem}

\newcommand{\p}[1]{\ensuremath{\mathord{\left(#1\right)}}}
\newcommand{\abs}[1]{\ensuremath{\mathord{\left\vert#1\right\vert}}}
\newcommand{\clint}[2]{\ensuremath{\left[#1, #2\right]}}
\newcommand{\downto}{\ensuremath{\downarrow}}
\newcommand{\RR}{\ensuremath{\mathbb{R}}}
\newcommand{\RRc}{\ensuremath{\overline{\RR}}}
\newcommand{\defeq}{\ensuremath{:=}}
\newcommand{\set}[1]{\ensuremath{\left\lbrace #1 \right\rbrace}}
\newcommand{\setcond}[2]{\ensuremath{\left\lbrace #1 \,\middle\vert\, #2 \right\rbrace}}
\newcommand{\weakto}{\ensuremath{\rightharpoonup}}
\let\subset\subseteq

\DeclareMathOperator{\Prox}{Prox}
\DeclareMathOperator{\Proj}{Proj}

\theoremstyle{plain}
\newtheorem{lemma}{Lemma}
\newtheorem{theorem}{Theorem}
\newtheorem{corollary}{Corollary}

\theoremstyle{remark}
\newtheorem*{remark}{Remark}

\theoremstyle{definition}
\newtheorem{example}{Example}
\newtheorem*{msc}{2010 Mathematics Subject Classification}
\newtheorem*{keys}{Keywords}

\title{Backward-Backward Splitting in Hadamard Spaces}
\author{Sebastian Banert\thanks{University of Vienna, Faculty of Mathematics, Oskar-Morgenstern-Platz 1, 1090 Vienna, Austria; sebastian.banert@univie.ac.at}}

\begin{document}
\maketitle
\begin{abstract}
  The backward-backward algorithm is a tool for finding minima of a regularization of the sum of two convex functions in Hilbert spaces. We generalize this setting to Hadamard spaces and prove the convergence of an error-tolerant version of the backward-backward method.
\end{abstract}
\begin{keys}
  Algorithm, Backward-backward splitting, CAT(0) space, Convex optimization, Hadamard space, Proximal point, Nonexpansive mapping, Weak convergence.
\end{keys}
\begin{msc}
  30L99, 65K05, 90C25
\end{msc}

\section{Introduction}
Proximal splitting methods provide powerful techniques for solving non-differentiable convex optimization problems in Hilbert spaces, see e.g. \cite{CombettesPesquet:2011} for a survey on this topic in the context of signal processing.

Recently, Ba\v{c}\'{a}k et al. \cite{BacakSearstonSims:2012, Bacak:2013} investigated the convergence of the proximal point algorithm and the alternating projection method for convex functions in Hadamard spaces, which are also known as complete CAT(0) spaces. The aim of this work is to generalize both approaches and to give the first example of a proximal splitting method in an Hadamard space.

This work is organized as follows: In Section \ref{sec:Preliminaries}, we set up our terminology. Particular attention is given to the geometry of Hadamard spaces (Section \ref{subsec:Prelims:Geometry}), where we mention inequality \eqref{eq:quad_triangle_inequality}, which is a stronger form of the triangle inequality in CAT(0) spaces. It enables us to generalize some well-known facts from Hilbert to Hadamard spaces, though its proof is elementary. In Section \ref{subsec:Prelims:Convexity}, the emphasis rests on convex functions, where our terminology is adopted from e.g. \cite{BauschkeCombettes:2011} in the context of Hilbert spaces and might be unusual in the community of CAT(0) spaces. Section \ref{subsec:Prelims:Weak} is devoted to weak convergence. For the history of generalizing weak convergence from Hilbert to Hadamard spaces, see e.g. \cite[Section 2.3]{BacakSearstonSims:2012}. In Section \ref{sec:algorithm}, we present a convergence analysis of the backward-backward algorithm and show its tolerance with respect to summable error sequences.

\section{Preliminaries}\label{sec:Preliminaries}
\subsection{Geometry of an Hadamard space}\label{subsec:Prelims:Geometry}
An \emph{Hadamard space} $\p{X, d}$ is a complete metric space, where to each two points $x, y\in X$ a \emph{midpoint} $m\in X$ can be assigned such that
\begin{equation}\label{eq:Hadamard:midpoint}
  d\p{z, m}^2 \leq \frac{1}{2} d\p{z, x}^2 + \frac{1}{2} d\p{z, y}^2 - \frac{1}{4} d\p{x, y}^2
\end{equation}
for all $z\in X$. If $X$ is a closed, convex subset of a Hilbert space with the metric induced by the inner product, relation \eqref{eq:Hadamard:midpoint} holds with equality for $m = \frac{1}{2}\p{x + y}$ and all $z\in X$ by the parallelogram identity.

More generally, for each two points $x, y\in X$, there is a map $\gamma_{x, y}: \clint{0}{1} \to X$, such that
\begin{equation}\label{eq:Hadamard:geodesic}
  d\p{z, \gamma_{x, y}\p{\lambda}}^2 \leq \p{1-\lambda} d\p{z, x}^2 + \lambda d\p{z, y}^2 - \lambda \p{1-\lambda} d\p{x, y}^2
\end{equation}
for all $z\in X$ and $\lambda\in\clint{0}{1}$. The curve $\gamma_{x, y}$ is uniquely determined and called the \emph{geodesic} joining $x$ and $y$. It holds
\[
  d\p{\gamma_{x, y}\p{\lambda_1}, \gamma_{x, y}\p{\lambda_2}} = \abs{\lambda_1 - \lambda_2} d\p{x, y}
\]
for $0 \leq \lambda_1 \leq \lambda_2 \leq 1$, $\gamma_{x, y}\p{0} = x$ and $\gamma_{x, y}\p{1} = y$.
The \emph{geodesic segment} joining $x\in X$ and $y\in X$ is defined as
\[
  \clint{x}{y} \defeq \setcond{\gamma_{x, y}\p{\lambda}}{0 \leq \lambda \leq 1}.
\]

From \eqref{eq:Hadamard:geodesic} (or \eqref{eq:Hadamard:midpoint}) one obtains a useful inequality by
\begin{align*}
  0 &\leq d\p{\gamma_{x, y}\p{\frac{1}{2}}, \gamma_{z, w}\p{\frac{1}{2}}}^2 \\
  &\leq \frac{1}{2} d\p{\gamma_{x, y}\p{\frac{1}{2}}, z}^2 + \frac{1}{2} d\p{\gamma_{x, y}\p{\frac{1}{2}}, w}^2 - \frac{1}{4}d\p{z, w}^2 \\
  &\leq \frac{1}{4}\p{d\p{x, z}^2 + d\p{y, z}^2 + d\p{x, w}^2 + d\p{y, w}^2 - d\p{x, y}^2 - d\p{z, w}^2}
\end{align*}
for all $x, y, z, w\in X$, which yields
\begin{equation}\label{eq:quad_triangle_inequality}
  d\p{x, y}^2 + d\p{z, w}^2 \leq d\p{x, z}^2 + d\p{x, w}^2 + d\p{y, z}^2 + d\p{y, w}^2.
\end{equation}
For two Hadamard spaces $X$ and $Y$, the Cartesian product $X \times Y$ is an Hadamard space with the metric given by
\[
  d\p{\p{x_1, y_1}, \p{x_2, y_2}}^2 = d\p{x_1, x_2}^2 + d\p{y_1, y_2}^2
\]
and the geodesics
\[
  \gamma_{\p{x_1, y_1}, \p{x_2, y_2}}\p{\lambda} = \p{\gamma_{x_1, x_2}\p{\lambda}, \gamma_{y_1, y_2}\p{\lambda}}
\]
for $x_1, x_2 \in X$ and $y_1, y_2\in Y$. We shall write $X^2$ for $X\times X$.

In what follows, let $X$ be an Hadamard space.

\subsection{Convexity, proximal points and firm nonexpansiveness}\label{subsec:Prelims:Convexity}
A set $C\subseteq X$ is called \emph{convex} if it contains the geodesics between all of its points, i.e., $\gamma_{x, y}\p{\lambda} \in C$ for all $x, y\in C$ and $\lambda \in \clint{0}{1}$. A map $f: X\to \RRc \defeq \RR \cup \set{\pm\infty}$ is called \emph{proper} if it is not constantly $+\infty$ and does not take the value $-\infty$. It is called \emph{convex} if, for any $x, y\in X$ and $\lambda\in\clint{0}{1}$, the inequality
\[
  f\p{\gamma_{x, y}\p{\lambda}} \leq \p{1-\lambda} f\p{x} + \lambda f\p{y}
\]
holds. If $f: X\to \RRc$ is convex, proper and lower semicontinuous and $x\in X$, then the function $y \mapsto f\p{y} + \frac{1}{2} d\p{x, y}^2$ has a unique minimizer \cite[Lemma 2]{Jost:1995}, which will be denoted by $\Prox_f\p{x}$, the \emph{proximal point} of $f$ corresponding to $x$. If $y = \Prox_f\p{x}$, $0 < \lambda \leq 1$ and $z\in X$, then it holds
\begin{align*}
  f\p{y} &\leq f\p{\gamma_{y, z}\p{\lambda}} + \frac{1}{2} d\p{x, \gamma_{y, z}\p{\lambda}}^2 - \frac{1}{2} d\p{x, y}^2 \\
  &\leq \p{1-\lambda} f\p{y} + \lambda f\p{z} \\ &\mathrel{\phantom{\leq}}\mathop{+} \frac{1}{2}\p{\p{1-\lambda} d\p{x, y}^2 + \lambda d\p{x, z}^2 - \lambda\p{1-\lambda} d\p{y, z}^2 - d\p{x, y}^2},
\end{align*}
so
\[
  \lambda f\p{y} \leq \lambda f\p{z} + \frac{1}{2}\p{-\lambda d\p{x, y}^2 + \lambda d\p{x, z}^2 - \lambda\p{1-\lambda} d\p{y, z}^2}.
\]
Dividing by $\lambda$ and letting $\lambda \downto 0$ yields
\begin{equation}\label{eq:prox_fval}
  f\p{y} \leq f\p{z} + \frac{1}{2} d\p{x, z}^2 - \frac{1}{2} d\p{x, y}^2 - \frac{1}{2}d\p{y, z}^2.
\end{equation}
For $x_1, x_2\in X$ and $y_i = \Prox_f\p{x_i}$, $i = 1, 2$, we have, by \eqref{eq:prox_fval},
\begin{align*}
  f\p{y_2} &\leq f\p{y_1} + \frac{1}{2}\p{d\p{x_2, y_1}^2 - d\p{x_2, y_2}^2 - d\p{y_2, y_1}^2}, \\
  f\p{y_1} &\leq f\p{y_2} + \frac{1}{2}\p{d\p{x_1, y_2}^2 - d\p{x_1, y_1}^2 - d\p{y_1, y_2}^2}.
\end{align*}
By adding both inequalities, we obtain
\begin{equation}\label{eq:firmly_nonexpansive}
  d\p{y_1, y_2}^2 \leq \frac{1}{2}\p{d\p{x_1, y_2}^2 + d\p{x_2, y_1}^2 - d\p{x_1, y_1}^2 - d\p{x_2, y_2}^2},
\end{equation}
which we refer to by saying that the mapping $x\mapsto \Prox_f\p{x}$ is \emph{firmly nonexpansive}.

In particular, a firmly nonexpansive mapping is \emph{nonexpansive}, i.e., Lipschitz continuous with constant $1$: let $x_1, x_2, y_1, y_2$ satisfy \eqref{eq:firmly_nonexpansive}, then from \eqref{eq:quad_triangle_inequality}, it follows that
\begin{align*}
  &\mathop{\phantom{\leq}}d\p{y_1, y_2}^2 \\
  &\leq \frac{1}{2}\p{d\p{x_1, x_2}^2 + d\p{x_1, y_1}^2 + d\p{y_2, x_2}^2 + d\p{y_2, y_1}^2 - d\p{x_1, y_1}^2 - d\p{x_2, y_2}^2} \\
  &= \frac{1}{2}\p{d\p{x_1, x_2}^2 + d\p{y_1, y_2}^2},
\end{align*}
so $d\p{y_1, y_2} \leq d\p{x_1, x_2}$.

\begin{example}\label{ex:indicator}
  For a nonempty, closed, convex set $C\subseteq X$, the \emph{indicator function}
  \[
    \delta_C: X\to \RRc, \qquad \delta_C\p{x} = \begin{cases}0& \text{if } x\in C, \\ +\infty& \text{otherwise,}\end{cases}
  \]
  is proper, convex and lower semicontinuous, and $\Prox_{\delta_C}\p{x} = \Proj_C\p{x}$ for all $x\in X$, where $\Proj_C$ is the \emph{metric projection} on C, i.e., $\Proj_C \p{x}$ uniquely minimizes the function $y\mapsto d\p{x, y}^2$ over $y\in C$.
\end{example}

\begin{example}\label{ex:distsquare}
  The function which maps $\p{x, y} \in X^2$ on $d\p{x, y}^2$ is convex. Indeed, we have, by \eqref{eq:Hadamard:geodesic} and \eqref{eq:quad_triangle_inequality},
  \begin{align*}
    d\p{\gamma_{x_1, x_2}\p{\lambda}, \gamma_{y_1, y_2}\p{\lambda}}^2 &\leq \p{1-\lambda}^2 d\p{x_1, y_1}^2 + \lambda^2 d\p{x_2, y_2}^2  \\ &\mathrel{\phantom{\leq}}\mathop{+} \lambda\p{1-\lambda}\p{d\p{x_2, y_1}^2 + d\p{y_2, x_1}^2 - d\p{y_1, y_2}^2 - d\p{x_1, x_2}^2} \\
    &\leq \p{1-\lambda}^2 d\p{x_1, y_1}^2 + \lambda^2 d\p{x_2, y_2}^2 \\ &\mathrel{\phantom{\leq}}\mathop{+} \lambda\p{1-\lambda} \p{d\p{x_1, y_1}^2 + d\p{x_2, y_2}^2} \\
    &= \p{1-\lambda} d\p{x_1, y_1}^2 + \lambda d\p{x_2, y_2}^2
  \end{align*}
  for all $x_1, x_2, y_1, y_2\in X$.
\end{example}

A function $f: X\to \RRc$ is called \emph{uniformly convex} if there exists some nondecreasing function $\phi: \left[0, +\infty\right) \to \clint{0}{+\infty}$ (which is called the \emph{modulus} of uniform convexity) such that $\phi\p{t} = 0 \iff t = 0$ ($t\geq 0$) and
  \[
    f\p{\gamma_{x, y}\p{\lambda}} \leq \p{1-\lambda} f\p{x} + \lambda f\p{y} - \lambda\p{1-\lambda} \phi\p{d\p{x, y}}
  \]
  for all $x, y\in X$. In this case, inequality \eqref{eq:prox_fval} can be improved: let $x, z\in X$, $\lambda \in \clint{0}{1}$ and $y = \Prox_f\p{x}$, then it holds
  \begin{align*}
    f\p{y} &\leq f\p{\gamma_{y, z}\p{\lambda}} + \frac{1}{2} d\p{x, \gamma_{y, z}\p{\lambda}}^2 - \frac{1}{2} d\p{x, y}^2 \\
    &\leq \p{1-\lambda} f\p{y} + \lambda f\p{z} - \lambda\p{1-\lambda} \phi\p{d\p{y, z}} \\
    &\mathrel{\phantom{\leq}}\mathop{+} \frac{1}{2}\p{\p{1-\lambda} d\p{x, y}^2 + \lambda d\p{x, z}^2 - \lambda\p{1-\lambda} d\p{y, z}^2 - d\p{x, y}^2},
  \end{align*}
  so, by a calculation as above,
  \begin{equation}\label{eq:fval_uniform}
    f\p{y} \leq f\p{z} - \phi\p{d\p{y, z}} + \frac{1}{2}\p{d\p{x, z}^2 - d\p{x, y}^2 - d\p{y, z}^2}.
  \end{equation}

  \subsection{Weak convergence}\label{subsec:Prelims:Weak}
  We say that a sequence $\p{x_n}_{n\geq 0}$ in $X$ \emph{weakly converges} to $x\in X$, $x_n \weakto x$, if, for every $y\in X$, $\Proj_{\clint{x}{y}}\p{x_n} \to x$. A \emph{weak cluster point} of a sequence in $X$ is a point $x\in X$ such that some subsequence weakly converges to $x$. (Since we are not dealing with nets, we will not distinguish between weak cluster points, which are determined by convergent subnets, and weak sequential cluster points.)

  \begin{lemma}[{see \cite[Theorem 2.1]{Jost:1994}}]\label{lem:subseqexistence}
    Every bounded sequence has a weakly convergent subsequence.
  \end{lemma}
  
  \begin{lemma}[{see \cite[Lemma 3.1]{Bacak:2013}}]\label{lem:weaklowersemicty}
    A convex, lower semicontinuous function is weakly lower semicontinuous, i.e., for a sequence $\p{x_n}_{n\geq 0}$ in $X$ with $x_n \weakto x \in X$, it holds
    \[
      \liminf_{n\to\infty} f\p{x_n} \geq f\p{x}.
    \]
  \end{lemma}
  
  Weak convergence may also be expressed in terms of asymptotic radius and center. For a sequence $\p{x_n}_{n\geq 0}$ in $X$ and $x\in X$, the \emph{asymptotic radius} of $\p{x_n}_{n\geq 0}$ with respect to $x$ is
  \[
    r\p{\p{x_n}_{n\geq 0}, x} \defeq \limsup_{n\geq 0} d\p{x_n, x},
  \]
  and the \emph{asymptotic center} is a minimizer of the mapping $X\ni x \mapsto r\p{\p{x_n}_{n\geq 0}, x}$. The asymptotic center of a sequence in an Hadamard space always exists and is unique \cite[Proposition 7]{DhompongsaKirkSims:2006}.

  \begin{lemma}[{see \cite[Proposition 5.2]{EspinolaFernandez-Leon:2009}}]
    For a sequence $\p{x_n}_{n\geq 0}$ in $X$ and $x\in X$ it holds $x_n \weakto x$ if and only if $x$ is the asymptotic center of each subsequence of $\p{x_n}_{n\geq 0}$.
  \end{lemma}

  \begin{lemma}\label{lem:wscpoint}
    Let $\p{x_n}_{n\geq 0}$ be a bounded sequence in $X$ and $x\in X$. Then $x_n \weakto x$ if and only if $x$ is the unique weak cluster point of $X$.
  \end{lemma}
  \begin{proof}
    We only prove the nontrivial implication: assume $x_n \not\weakto x$, then there exists some $y\in X\setminus\set{x}$ and a subsequence $\p{x_{n_k}}_{k\geq 0}$ of which $y$ is the asymptotic center, in particular
    \[
      \limsup_{k\to\infty} d\p{x_{n_k}, x} > \limsup_{n\to\infty} d\p{x_{n_k}, y}.
    \]
    By transition to a subsequence (without renaming), we can assure
    \[
      \lim_{k\to\infty} d\p{x_{n_k}, x} > \limsup_{n\to\infty} d\p{x_{n_k}, y}.
    \]
    Now, choose a weakly convergent subsequence of the bounded sequence $\p{x_{n_k}}_{k\geq 0}$. Its asymptotic center cannot be $x$ (since the asymptotic radius with respect to $y$ is smaller), so $x$ cannot be the only weak cluster point of $\p{x_n}_{n\geq 0}$.
  \end{proof}

  The next result is a generalization of \cite[Proposition 3.3 (iii)]{BacakSearstonSims:2012} (see also \cite[Lemma 2.39]{BauschkeCombettes:2011}), where Fej\'er monotonicity is required to show weak convergence. To analyze an algorithm which is generally not Fej\'er monotone, we have to relax this assumption, but the method of the proof remains the same.
  \begin{lemma}\label{lem:weakconv}
    Let $\p{x_n}_{n\geq 0}$ be a bounded sequence in $X$, and let $C\subseteq X$. Suppose that, for every $c\in C$, the sequence $\p{d\p{x_n, c}}_{n\geq 0}$ converges and all weak cluster points of $\p{x_n}_{n\geq 0}$ belong to $C$. Then, $x_n \weakto c$ for some $c\in C$.
  \end{lemma}
  \begin{proof}
    By Lemma \ref{lem:wscpoint}, it suffices to show that $\p{x_n}_{n\geq 0}$ has at most one weak cluster point, so assume that $\p{x_{n_k}}_{k\geq 0}$ and $\p{x_{m_k}}_{k\geq 0}$ are weakly convergent subsequences with $x_{n_k} \weakto x\in C$, $x_{m_k} \weakto y\in C$ and $x\neq y$. In particular, $x$ is the asymptotic center of $\p{x_{n_k}}_{n\geq 0}$, and $y$ is the asymptotic center of $\p{x_{m_k}}_{k\geq 0}$. By uniqueness of the asymptotic centers, we have
    \begin{align*}
      \lim_{n\to \infty} d\p{x_n, x} &= \lim_{k\to\infty} d\p{x_{n_k}, x} < \limsup_{k\to\infty} d\p{x_{n_k}, y} = \lim_{n\to\infty} d\p{x_n, y} = \lim_{k\to\infty} d\p{x_{m_k}, y} \\ &< \limsup_{k\to\infty} d\p{x_{m_k}, x} = \lim_{n\to\infty} d\p{x_n, x}.
    \end{align*}
    This contradiction shows $x = y$.
  \end{proof}

  The notion of weak convergence in Hadamard spaces generalizes the weak convergence in Hilbert spaces. It is also known as \emph{$\Delta$-convergence} \cite{Lim:1976}.

  \section{The backward-backward algorithm and its convergence}\label{sec:algorithm}
  Let $X$ be an Hadamard space, and let $f, g: X\to \RRc$ be two proper, convex and lower semicontinuous functions. Our aim is to find a solution of the problem
  \begin{equation}\label{eq:problem:bb}
    \text{minimize } \Phi\p{x, y} \defeq f\p{x} + g\p{y} + \frac{1}{2\gamma} d\p{x, y}^2\qquad \text{over } x, y\in X,
  \end{equation}
  where $\gamma > 0$. The \emph{backward-backward algorithm} is determined by some starting point $x_0\in X$ and the iteration procedure $y_n \defeq \Prox_{\gamma g} x_n$, $x_{n+1} \defeq \Prox_{\gamma f} y_n$ for all $n\geq 0$.

  Problem \eqref{eq:problem:bb} and the backward-backward algorithm in Hilbert spaces are dealt with in \cite{AckerPrestel:1980, BauschkeCombettesReich:2005}. In this work, we allow errors in the evalution of the proximal points, namely, we choose the sequences $\p{x_n}_{n\geq 0}$ and $\p{y_n}_{n\geq 0}$ such that
  \begin{equation}\label{eq:bb:errors}
    \sum_{n=0}^\infty d\p{y_n, \Prox_{\gamma g} x_n} < +\infty \qquad \text{and} \qquad \sum_{n=0}^\infty d\p{x_{n+1}, \Prox_{\gamma f} y_n} < +\infty.
  \end{equation}
  Additional consequences of the error-free case are given in Corollary \ref{cor:errorfree} below.
  
  \begin{theorem}\label{thm:bb}
    Let $\Phi$ be bounded below on $X^2$, and let $\p{x_n}_{n\geq 0}$ and $\p{y_n}_{n\geq 0}$ be sequences which satisfy \eqref{eq:bb:errors}. Then the following hold:
    \begin{enumerate}[label=(\alph*)]
      \item \label{item:thm:bb:summable} $\displaystyle\sum_{n = 0}^\infty d\p{x_{n+1}, x_n}^2 < +\infty$ and $\displaystyle\sum_{n=0}^\infty d\p{y_{n+1}, y_n}^2 < +\infty$;
    \end{enumerate}
    if there exists a solution of \eqref{eq:problem:bb}, then
    \begin{enumerate}[label=(\alph*), resume]
      \item \label{item:thm:bb:err_fval}
        \begin{center}
        \begin{minipage}{\linewidth}
          \begin{align*}
            \inf\setcond{\Phi\p{x, y}}{x, y\in X} &= \lim_{n\to\infty} \Phi\p{\Prox_{\gamma f}\p{\Prox_{\gamma g}\p{x_n}}, \Prox_{\gamma g}\p{x_n}} \\
            &= \lim_{n\to\infty} \Phi\p{\Prox_{\gamma f}\p{y_n}, \Prox_{\gamma g}\p{\Prox_{\gamma f}\p{y_n}}};
          \end{align*}
        \end{minipage}
      \end{center}
    \item \label{item:thm:bb:weakconv} $\p{\p{x_n, y_n}}_{n\geq 0}$ converges weakly to a solution of \eqref{eq:problem:bb};
    \item \label{item:thm:bb:uniform} if one of the functions $f$ and $g$ is uniformly monotone, then $\p{\p{x_n, y_n}}_{n\geq 0}$ converges to the unique solution of \eqref{eq:problem:bb} with respect to the metric $d$.
    \end{enumerate}
  \end{theorem}

  \begin{proof}
    Due to the symmetry between $f$ and $g$, we will prove only one of the assertions in each of the statements \ref{item:thm:bb:summable}--\ref{item:thm:bb:uniform}, mostly the one concerning the sequence $\p{x_n}_{n\geq 0}$. The other assertion follows by interchanging $f$ and $g$ and starting the iteration at $y_0$.

    Set $\delta_n \defeq d\p{y_n, \Prox_{\gamma g} x_n}$ and $\varepsilon_n \defeq d\p{x_{n+1}, \Prox_{\gamma f} y_n}$. Furthermore, for $n\geq 0$, set $\tilde x_0 \defeq x_0$, $\tilde y_0 \defeq y_0$,
    \[
      \tilde x_{n+1} \defeq \Prox_{\gamma f}\p{\Prox_{\gamma g}\p{x_n}}, \qquad \tilde y_{n+1} \defeq \Prox_{\gamma g}\p{\Prox_{\gamma f}\p{y_n}}.
    \]

    By the nonexpansiveness of the proximal point mapping, we have
    \begin{align*}
      d\p{\tilde x_{n+1}, x_{n+1}} &\leq d\p{\Prox_{\gamma f}\p{\Prox_{\gamma g}\p{x_n}}, \Prox_{\gamma f}\p{y_n}} + d\p{\Prox_{\gamma f}\p{y_n}, x_{n+1}} \\
      &\leq d\p{\Prox_{\gamma g}\p{x_n}, y_n} + \varepsilon_n \\
      &= \delta_n + \varepsilon_n
    \end{align*}
    and analogously $d\p{\tilde y_{n+1}, y_{n+1}} \leq \delta_n + \varepsilon_n$.

    By \eqref{eq:prox_fval}, the following inequalities hold for all $x, y\in X$:
    \begin{align}
      f\p{\tilde x_{n+1}} &\leq f\p{x} + \frac{1}{2\gamma}\p{d\p{\Prox_{\gamma g}\p{x_n}, x}^2 - d\p{\Prox_{\gamma g}\p{x_n}, \tilde x_{n+1}}^2 - d\p{\tilde x_{n+1}, x}^2}, \label{eq:err:fval:f1} \\
      g\p{\Prox_{\gamma g}\p{x_n}} &\leq g\p{y} + \frac{1}{2\gamma}\p{d\p{x_n, y}^2 - d\p{x_n, \Prox_{\gamma g}\p{x_n}}^2 - d\p{\Prox_{\gamma g}\p{x_n}, y}^2}, \label{eq:err:fval:g1}
    \end{align}
    Adding \eqref{eq:err:fval:f1} and \eqref{eq:err:fval:g1} and subsequently applying \eqref{eq:quad_triangle_inequality}, we obtain
    \begin{align}
      \Phi\p{\tilde x_{n+1}, \Prox_{\gamma g}\p{x_n}} &\leq \Phi\p{x, y} + \frac{1}{2\gamma}\Bigl(d\p{\Prox_{\gamma g}\p{x_n}, x}^2 + d\p{x_n, y}^2 - d\p{x, y}^2 \nonumber \\ &\mathrel{\phantom{=}}\mathop{-} d\p{x_n, \Prox_{\gamma g}\p{x_n}}^2 - d\p{\tilde x_{n+1}, x}^2 - d\p{\Prox_{\gamma g}\p{x_n}, y}^2\Bigr) \nonumber \\
      &\leq \Phi\p{x, y} + \frac{1}{2\gamma}\p{d\p{x_n, x}^2 - d\p{\tilde x_{n+1}, x}^2}. \label{eq:err:fundineq:x}
    \end{align}
    Setting $x = \tilde x_n$ and $y = \Prox_{\gamma g}\p{x_{n-1}}$ in \eqref{eq:err:fundineq:x} yields
    \begin{align}
      \Phi\p{\tilde x_{n+1}, \Prox_{\gamma g}\p{x_n}} &\leq \Phi\p{\tilde x_n, \Prox_{\gamma g}\p{x_{n-1}}} + \frac{1}{2\gamma}\p{d\p{x_n, \tilde x_n}^2 - d\p{\tilde x_{n+1}, \tilde x_n}^2}. \label{eq:err:telescope:x}
    \end{align}

    Summing up the inequality \eqref{eq:err:telescope:x} for $n = 1, \ldots, N-1$ with $N\geq 1$ gives
    \begin{align*}
      \sum_{n=0}^{N-1} d\p{x_{n+1}, x_n}^2 &\leq \sum_{n=0}^{N-1} \p{d\p{x_{n+1}, \tilde x_{n+1}} + d\p{\tilde x_{n+1}, \tilde x_n} + d\p{\tilde x_n, x_n}}^2 \\
      &\leq 3 \sum_{n=0}^{N-1} \p{d\p{x_{n+1}, \tilde x_{n+1}}^2 + d\p{\tilde x_{n+1}, \tilde x_n}^2 + d\p{\tilde x_n, x_n}^2} \\
      &\leq 6 \sum_{n=0}^\infty \p{\delta_n + \varepsilon_n}^2 + 3\sum_{n=0}^{N-1} d\p{\tilde x_{n+1}, \tilde x_n}^2 \\
      &\leq 9 \sum_{n=0}^\infty \p{\delta_n + \varepsilon_n}^2 + 6\gamma \p{\Phi\p{\tilde x_1, \Prox_{\gamma g}\p{x_0}} - \Phi\p{\tilde x_N, \Prox_{\gamma g}\p{x_{N-1}}}},
    \end{align*}
    which shows \ref{item:thm:bb:summable}.

    From now on, let $\p{\bar x, \bar y} \in X\times X$ be a solution of \eqref{eq:problem:bb}, which implies $\bar y = \Prox_{\gamma g}\p{\bar x}$ and $\bar x = \Prox_{\gamma f}\p{\bar y}$, and let
    \[
      \ell \defeq \inf\setcond{\Phi\p{x, y}}{x, y \in X} = \Phi\p{\bar x, \bar y}.
    \]
    be the optimal value of $\Phi$.
    We have
    \begin{align*}
      d\p{x_{n+1}, \bar x} &\leq d\p{\tilde x_{n+1}, \bar x} + d\p{\tilde x_{n+1}, x_{n+1}} \\
      &\leq d\p{\Prox_{\gamma f}\p{\Prox_{\gamma g}\p{x_n}}, \Prox_{\gamma f}\p{\Prox_{\gamma g}\p{\bar x}}} + \delta_n + \varepsilon_n \\
      &\leq d\p{x_n, \bar x} + \delta_n + \varepsilon_n.
    \end{align*}
    Therefore, the sequence 
    \[
      \p{d\p{x_n, \bar x} - \sum_{k=0}^{n-1} \p{\delta_k + \varepsilon_k}}_{n\geq 0}
    \]
    is monotone decreasing and bounded below by $-\sum_{k=0}^\infty \p{\delta_k + \varepsilon_k}$, thus convergent. From \eqref{eq:bb:errors}, it follows that the sequence $\p{d\p{x_n, \bar x}}_{n\geq 0}$ is convergent, say $d\p{x_n, \bar x} \to \xi$. Since $d\p{x_n, \tilde x_n} \leq \delta_{n-1} + \varepsilon_{n-1} \to 0$, we also have $d\p{\tilde x_n, \bar x} \to \xi$.

    In \eqref{eq:err:fundineq:x}, set $x = \bar x$ and $y = \bar y$ to obtain
    \begin{equation}\label{eq:err:solution}
      \ell \leq \Phi\p{\tilde x_{n+1}, \Prox_{\gamma g}\p{x_n}} \leq \ell + \frac{1}{2\gamma} \p{d\p{x_n, \bar x}^2 - d\p{\tilde x_{n+1}, \bar x}^2}.
    \end{equation}
    The right-hand side converges to $\ell + \frac{1}{2\gamma}\p{\xi^2 - \xi^2} = \ell$, so $\Phi\p{\tilde x_{n+1}, \Prox_{\gamma g}\p{x_n}} \to \ell$, and \ref{item:thm:bb:err_fval} holds.

    To prove \ref{item:thm:bb:weakconv}, it remains (by Lemma \ref{lem:weakconv}) to show that each weak cluster point of $\p{x_n, y_n}$ is a solution of \eqref{eq:problem:bb}. Let $\p{x, y}$ be a weak cluster point of $\p{x_n, y_n}$, say $x_{n_k} \weakto x$ and $y_{n_k} \weakto y$. Since $d\p{\tilde x_{n_k}, x_{n_k}} \leq \delta_{n_k-1} + \varepsilon_{n_k-1} \to 0$ and $d\p{\Prox_{\gamma g}\p{x_{n_k}}, y_{n_k}} = \delta_n \to 0$, it also holds $\tilde x_{n_k} \weakto x$ and $\Prox_{\gamma g} x_{n_k} \weakto y$.

    By the weak lower semicontinuity (see Lemma \ref{lem:weaklowersemicty} and Example \ref{ex:distsquare}) of the functions $f, g$ and $\p{x, y} \mapsto d\p{x, y}^2$, one gets
    \begin{align*}
      \Phi\p{x, y} 
      &= f\p{x} + g\p{y} + \frac{1}{2\gamma} d\p{x, y}^2 \\
      &\leq \liminf_{k\to\infty} f\p{\tilde x_{n_k}} + \liminf_{k\to\infty} g\p{\Prox_{\gamma g} x_{n_k}} + \frac{1}{2\gamma} \liminf_{k\to\infty} d\p{\tilde x_{n_k}, \Prox_{\gamma g} x_{n_k}}^2 \\
      &\leq \liminf_{k\to\infty} \p{f\p{\tilde x_{n_k}} + g\p{\Prox_{\gamma g} x_{n_k}} + \frac{1}{2\gamma} d\p{\tilde x_{n_k}, \Prox_{\gamma g} x_{n_k}}^2} \\
      &= \liminf_{k\to\infty} \Phi\p{\tilde x_{n_k}, \Prox_{\gamma g} x_{n_k}} = \ell,
    \end{align*}
    so $\p{x, y}$ is a solution of \eqref{eq:problem:bb}, which proves \ref{item:thm:bb:weakconv}.

    Now, let $f$ be uniformly convex with modulus $\phi$, so, by using \eqref{eq:fval_uniform} instead of \eqref{eq:prox_fval}, a calculation analogous to \eqref{eq:err:solution} yields
    \[
      \Phi\p{\tilde x_{n+1}, \Prox_{\gamma g}\p{x_n}} \leq \ell + \frac{1}{2\gamma} \p{d\p{x_n, \bar x}^2 - d\p{\tilde x_{n+1}, \bar x}^2} - \frac{1}{\gamma} \phi\p{d\p{\tilde x_{n+1}, \bar x}}
    \]
    or
    \[
      \phi\p{d\p{\tilde x_{n+1}, \bar x}} \leq \gamma\p{\ell - \Phi\p{\tilde x_{n+1}, \Prox_{\gamma g}\p{x_n}}} + \frac{1}{2} \p{d\p{x_n, \bar x}^2 - d\p{\tilde x_{n+1}, \bar x}^2}
    \]
    for any solution $\p{\bar x, \bar y}$ of \eqref{eq:problem:bb}. Suppose that $x_n \not\to \bar x$, then (since $d\p{x_n, \tilde x_n} \to 0$) $\tilde x_n \not\to \bar x$, and there exists $\varepsilon > 0$ such that for all $N\geq 1$ there exists some $n\geq N$ such that $d\p{\tilde x_n, \bar x} > \varepsilon$, which implies $\phi\p{d\p{\tilde x_n, \bar x}} \geq \phi\p{\varepsilon} > 0$. On the other hand,
    \[
      \phi\p{d\p{\tilde x_{n+1}, \bar x}} \leq \frac{1}{2} \p{d\p{x_n, \bar x}^2 - d\p{\tilde x_{n+1}, \bar x}^2} \to 0.
    \]
    This contradiction shows $x_n \to \bar x$. On the other hand,
    \begin{align*}
      d\p{y_n, \bar y}
      &\leq d\p{y_n, \Prox_{\gamma g}\p{x_n}} + d\p{\Prox_{\gamma g}\p{x_n}, \bar y} \\
      &= \delta_n + d\p{\Prox_{\gamma g}\p{x_n}, \Prox_{\gamma g}\p{\bar x}} \\
      &\leq \delta_n + d\p{x_n, \bar x}.
    \end{align*}
    Since $\delta_n \to 0$ and $x_n \to \bar x$, we have $y_n \to \bar y$. We have shown $x_n \to \bar x$ and $y_n \to \bar y$ for \emph{any} solution $\p{\bar x, \bar y}$ of \eqref{eq:problem:bb}, so the solution must be unique.
  \end{proof}

  \begin{remark}
    \begin{enumerate}[label=(\alph*)]
      \item Both the proximal point algorithm with a constant stepsize and the method of alternating projections are special cases of the backward-backward method presented here: for the former set $g\p{x} = 0$ for all $x\in X$ and for the latter (see Example \ref{ex:indicator}) set $f = \delta_C$ and $g = \delta_D$ with $C, D \subset X$ nonempty, closed and convex.
      \item In general, the backward-backward method cannot be expected to converge with respect to the metric $d$, since the proximal point algorithm (even if restricted to Hilbert spaces and constant stepsizes) is known to converge only weakly in general \cite[Corollary 5.1]{Gueler:1991}.
      \item The backward-backward algorithm in Hilbert spaces is a special case of the forward-backward algorithm \cite[Example 10.11]{CombettesPesquet:2011}, the error-tolerant version of which was considered in \cite{Vu:2013}.
    \end{enumerate}
  \end{remark}

  \begin{corollary}\label{cor:errorfree} Let $x_0 \in X$, and let $y_n \defeq \Prox_{\gamma g}\p{x_n}$ and $x_{n+1} \defeq \Prox_{\gamma f}\p{y_n}$ for $n\geq 0$. Then, additionally to the properties \ref{item:thm:bb:summable}--\ref{item:thm:bb:uniform} of Theorem \ref{thm:bb}, the following hold without further assumptions:
    \begin{enumerate}[label=(\alph*), start=5]
      \item \label{item:thm:bb:mondec} $\Phi\p{x_{n+1}, y_{n+1}} \leq \Phi\p{x_{n+1}, y_n} \leq \Phi\p{x_n, y_n}$ for $n\geq 1$,
      \item \label{item:thm:bb:fval} $\displaystyle\lim_{n\to\infty} \Phi\p{x_n, y_n} = \lim_{n\to\infty} \Phi\p{x_{n+1}, y_n} = \inf\setcond{\Phi\p{x, y}}{x, y\in X}$ (this value might be $-\infty$);
      \item \label{item:thm:bb:fejer} for each solution $\p{\bar x, \bar y}$ of \eqref{eq:problem:bb}, it holds
        \[
          d\p{x_{n+1}, \bar x} \leq d\p{y_n, \bar y} \leq d\p{x_n, \bar x}.
        \]
    \end{enumerate}
  \end{corollary}
  \begin{proof}
    Consider \eqref{eq:err:fundineq:x}, and note that now $\tilde x_{n+1} = x_{n+1}$ and $\Prox_{\gamma g}\p{x_n} = y_n$, so that we have
    \begin{equation}\label{eq:fundineq:x}
      \Phi\p{x_{n+1}, y_n} \leq \Phi\p{x, y} + \frac{1}{2\gamma} \p{d\p{x_n, x}^2 - d\p{x_{n+1}, x}^2}
    \end{equation}
    for any $x, y\in X$. Setting $x = x_n$ and $y = y_n$ gives
    \[
      \Phi\p{x_{n+1}, y_n} \leq \Phi\p{x_n, y_n} - \frac{1}{2\gamma} d\p{x_{n+1}, x_n}^2 \leq \Phi\p{x_n, y_n}.
    \]
    The other inequality in \ref{item:thm:bb:mondec} follows by interchanging $f$ and $g$.

    Note that by \ref{item:thm:bb:mondec} the limits in \ref{item:thm:bb:fval} both coincide with $\inf\setcond{\Phi\p{x_{n+1}, y_n}}{n\geq 0}$. Sum up \eqref{eq:fundineq:x} for $n = 0, \ldots, N-1$ with $N\geq 1$ and divide by $N$ to obtain
    \begin{align*}
      \inf\setcond{\Phi\p{x_{n+1}, y_n}}{n\geq 0} &\leq \frac{1}{N}\sum_{n=0}^{N-1} \Phi\p{x_{n+1}, y_n} \\
      &\leq \Phi\p{x, y} + \frac{1}{2\gamma N} \p{d\p{x_0, x}^2 - d\p{x_N, x}^2} \\
      &\leq \Phi\p{x, y} + \frac{1}{2\gamma N} d\p{x_0, x}^2.
    \end{align*}
    Letting $N\to +\infty$ and passing to the infimum over $x, y\in X$, we see
    \[
      \inf\setcond{\Phi\p{x_{n+1}, y_n}}{n\geq 0} \leq \inf\setcond{\Phi\p{x, y}}{x, y\in X}.
    \]
    Since the reverse inequality is obvious, \ref{item:thm:bb:fval} holds.

    Statement \ref{item:thm:bb:fejer} follows from $\bar y = \Prox_{\gamma g}\p{\bar x}$, $\bar x = \Prox_{\gamma f}\p{\bar y}$ and the nonexpansiveness of the proximal point mappings.
  \end{proof}

  \section*{Acknowledgements}
  The author is grateful to Radu Ioan Bo\c{t} for many helpful discussions and to the Department of Mathematics at Chemnitz University of Technology for providing the infrastructure at which a part of this work was created.
  \printbibliography
  \end{document}